\documentclass[a4paper]{amsart}
\usepackage{mathtools}
\usepackage{graphicx}
\usepackage{amssymb} 
\usepackage{stmaryrd}
\usepackage[mathcal]{euscript}
\usepackage{tikz}
\usepackage{tikz-cd}
\usepackage{hyperref}
\hypersetup{%
  bookmarksnumbered=true,%
  colorlinks=true,%
  linkcolor=blue,%
  citecolor=blue,%
  filecolor=blue,%
  menucolor=blue,%
  urlcolor=blue,%
  bookmarksopen=true,%
  bookmarksdepth=2,%
  pageanchor=true}

\makeatletter
\@namedef{subjclassname@2020}{\textup{2020} Mathematics Subject Classification}
\makeatother

\title{The category of finite strings}

\author{Henning Krause}
\address{Fakult\"at f\"ur Mathematik\\
Universit\"at Bielefeld\\ D-33501 Bielefeld\\ Germany}
\email{hkrause@math.uni-bielefeld.de}

\theoremstyle{plain}
\newtheorem{thm}{Theorem}[section]
\newtheorem{prop}[thm]{Proposition}
\newtheorem{lem}[thm]{Lemma} 

\newtheorem{cor}[thm]{Corollary}

\theoremstyle{definition}

\newtheorem{exm}[thm]{Example}

\theoremstyle{remark}
\newtheorem{rem}[thm]{Remark}

\numberwithin{equation}{section}

\newcommand{\add}{\operatorname{add}}

\newcommand{\card}{\operatorname{card}}

\newcommand{\End}{\operatorname{End}}

\newcommand{\Ext}{\operatorname{Ext}}

\newcommand{\Hom}{\operatorname{Hom}}
\newcommand{\HOM}{\operatorname{\mathcal{H}\!\!\;\mathit{om}}}
\newcommand{\id}{\operatorname{id}}

\renewcommand{\Im}{\operatorname{Im}}

\newcommand{\Ker}{\operatorname{Ker}}

\renewcommand{\mod}{\operatorname{mod}}

\newcommand{\NC}{\operatorname{NC}}

\newcommand{\Rep}{\operatorname{Rep}}

\newcommand{\Thick}{\operatorname{Thick}}

\newcommand{\op}{\mathrm{op}}

\newcommand{\Set}{\mathrm{Set}}

\newcommand{\iso}{\xrightarrow{\raisebox{-.4ex}[0ex][0ex]{$\scriptstyle{\sim}$}}}
\newcommand{\leftiso}{\xleftarrow{\raisebox{-.4ex}[0ex][0ex]{$\scriptstyle{\sim}$}}}

\newcommand{\longiso}{\xrightarrow{\ \raisebox{-.4ex}[0ex][0ex]{$\scriptstyle{\sim}$}\ }}

\newcommand{\lto}{\longrightarrow}

\newcommand{\xto}{\xrightarrow}
\newcommand*{\intref}[2]{\def\tmp{#1}\ifx\tmp\empty\hyperref[#2]{\ref*{#2}}\else\hyperref[#2]{#1~\ref*{#2}}\fi}

\def\A{\mathcal A} 
\def\B{\mathcal B} 
\def\C{\mathcal C}

\def\bfD{\mathbf D}

\def\bbN{\mathbb N}

\def\bbZ{\mathbb Z}

\def\a{\alpha}
\def\b{\beta}
\def\e{\varepsilon}
\def\d{\delta}

\def\p{\phi}

\def\s{\sigma}
\def\t{\tau}

\def\De{\Delta}

\def\Si{\Sigma}

\begin{document}



\begin{abstract}
 We introduce the category of finite strings and study its basic
 properties. The category is closely related to the augmented simplex
 category, and it models categories of linear representations. Each lattice of
 non-crossing partitions arises naturally  as a lattice of subobjects.
\end{abstract}
  
\date{13 September, 2021}

\maketitle

\section{Introduction}

Strings are considered to be one of the most basic combinatorial
structures arising in representation theory of associative
algebras. In fact, many of the interactions with neighbouring fields
involve strings and their corresponding representations, which are
also known as string modules.
% Examples arise from cluster theory, the theory of stability
% conditions, or homological mirror symmetry

In this note we introduce a category of finite strings and establish
some connections. First of all, we notice that the category of
connected strings is equivalent to the augmented simplex category
$\De$ (cf.\ \cite{GZ1967,ML1998}), once the initial and terminal
objects in $\De$ are identified. Then we show that the category of
finite strings models categories of linear representations. More
precisely, we provide an equivalence between finite strings and
certain abelian categories (hereditary and uniserial length categories
with only finitely many simple objects and split over a fixed field,
cf.\ \cite{AR1968}), where morphisms between strings correspond to
certain exact functors. In this context it is appropriate to include
cyclic strings which correspond to abelian categories of infinite
representation type. This is somewhat parallel to the cyclic category
of Connes and others \cite{Ca1985,Co1983}; however we add new objects
(cyclic strings) while the cyclic category keeps the objects of $\De$
and only morphisms are added.

Any morphism in the category of finite strings admits an epi-mono
factorisation. Thus it is of interest to study the subobjects of a
given object, at least for any connected string. We show for a linear
string of length $n$ that the lattice of subobjects is isomorphic to
the lattice $\NC(n+1)$ of non-crossing partitions \cite{Kr1972}, while
the lattice $\NC^B(n)$ of type $B$ non-crossing partitions 
\cite{Re1997} arises for a cyclic string of length $n$.

The correspondence between strings and categories of linear
representations identifies subobjects of strings with thick
subcategories of abelian categories. In this way we recover the
beautiful classification of thick subcategories for quiver
representations of type $A$ due to Ingalls and Thomas \cite{IT2009},
and we add a classification for nilpotent representations of cyclic
quivers which seems to be new. The cyclic case can be
used to complete the classification of all thick subcategories for
representations of any tame hereditary algebra, including the ones
that are not generated by exceptional sequences, and therefore
complementing the work in \cite{HK2016,IS2010,IT2009}. We
point out that the category of strings can be extended to include all
Dynkin types, beyond the type $A$ in this work, and analogous to the
categorification of non-crossing partitions for all Dynkin types in
\cite{HK2016}.

Finally, let us mention the connection with some recent work which is
concerned with Iyama's higher Auslander algebras of type $A$
\cite{I2011}. These algebras form a natural generalisation of the
hereditary algebras of type $A$ arising in the present work. In
\cite{DJW2019} the authors point out the simplicial structure of the
representations for these higher Auslander algebras, using some
advanced categorical formalism. Wide subcategories of representations
generalise thick subcategories and these are studied for type $A$
higher Auslander algebras in \cite{HJ2021}.

\subsubsection*{Acknowledgements} It is a pleasure to thank Marc
Stephan and Dieter Vossieck for several useful comments on this
work. In addition I wish to thank Christian Stump for pointing me to
the non-crossing partitions of type $B$.

\section{Connected strings}

In this work we introduce the category of finite strings. For any
natural number $n$ the connected string of length $n$ is
denoted by $\Si_n$. Each string comes equipped with its set of
(connected) substrings, together with a multiplication on the set of
substrings given by concatenation. The objects of the category are
finite coproducts of connected strings, and the morphisms are maps
that preserve substrings and their multiplication.

A \emph{basic string} is a pair $s=(s',s'')$ of integers $s'\le s''$.
We write $\ell(s)= s''-s'+1$ for the \emph{length} of $s$ and add a
\emph{zero string} $*$ satisfying $\ell(*)=0$. Strings of length one
are called \emph{simple} and we set $s_i:=(i,i)$, $i\in\bbZ$.  For a
string $s=(i,j)$ we call the simple strings $s_{i},s_{i+1},\ldots,s_{j}$ its
\emph{composition factors}.

A multiplication of basic strings is given by concatenation.  For
$s,t$ set
\[st:=\begin{cases}
    (s',t'')&\text{for } s=(s',s''),\, t=(t',t''),\, s''+1=t',\\
    t&\text{for } s=*,\\  s&\text{for } t=*,\\
    * &\text{otherwise}.
  \end{cases}\] This multiplication is not associative. For instance,
we have
\[(s_0 s_1)s_0=(0,1)s_0=*\qquad\text{and}\qquad s_0(s_1
  s_0)=s_0*=s_0.\]

For $n\in\bbN=\{0,1,2,\ldots\}$ the \emph{connected string} of length
$n$ is by definition the set of basic strings
\[\Si_n:=\{s=(s',s'')\mid 0\le s'\le s''< n\}\cup\{*\}.\]
A morphism $\p\colon\Si_m\to\Si_n$ is by definition a map such that
for all $s,t\in\Si_m$
\begin{equation}\label{eq:mor}
  \p(st)=\p(s)\p(t)\qquad\text{and}\qquad
  st\neq *=\p(st)\implies \p(s)=*=\p(t).
\end{equation}
The string $*$ plays the
role of a base point. In fact, a morphism is base point preserving
since \[\p(*)=\p(**)=\p(*)\p(*)=*.\] Any morphism is determined by the
images of the simple strings but it need not preserve the length of
basic strings.
 
We define \emph{standard morphisms} as follows. Let $n\ge 1$. The morphism
  \[\d^i_n\colon \Si_{n-1}\lto\Si_n \qquad (0\le i\le n)\] 
  is given by the unique injective map such that $s_{i-1}$ and $s_{i}$
  are not in its image. Note that $\d_1^0=\d_1^1$. The
  morphism \[\s^i_n\colon \Si_{n}\lto\Si_{n-1} \qquad (0\le i <n)\] is
  given by the unique surjective map such that $s_i$ is sent to the
  zero string.

  The standard morphisms are analogues of the face and degeneracy maps
  for simplices. In fact, they satisfy the following \emph{simplicial identities} \cite[VII.5]{ML1998}.

\begin{lem}\label{le:rel}
The standard morphisms satisfy the following identities:
\begin{equation}\label{eq:rel}
\begin{aligned}
  \d^i_{n+1}\circ\d^j_n&=\mathrlap{\d^{j+1}_{n+1}\circ\d^i_n}
 \hphantom{\begin{cases}\d^i_{n}\circ\s^{j-1}_{n}&\\ 
    \id &\\
    \d^{i-1}_{n}\circ\s^{j}_{n}&\end{cases}}\kern-\nulldelimiterspace
  i\le j\\
  \s^j_{n}\circ\s^i_{n+1}&= \mathrlap{\s^i_{n}\circ\s^{j+1}_{n+1}}
  \hphantom{\begin{cases}\d^i_{n}\circ\s^{j-1}_{n}&\\ 
    \id &\\
    \d^{i-1}_{n}\circ\s^{j}_{n}&\end{cases}}\kern-\nulldelimiterspace
                           i\le j\\
  \s^j_{n}\circ\d^i_{n}&=  \begin{cases}\d^i_{n}\circ\s^{j-1}_{n}& i<j\\ 
    \id &i=j\text{ or } i=j+1\\
    \d^{i-1}_{n}\circ\s^{j}_{n}& i>j+1\
\end{cases}
\end{aligned}
\end{equation}
\end{lem}
\begin{proof}
  This is easily checked, for instance by tracing the images of the
  simple strings. We have
  \begin{equation}\label{eq:standard}
    \d_n^i(s_j)=\begin{cases} s_j&j<i-1\\
      (j-1,j)&j=i-1\\
      s_{j+1}&j>i-1\end{cases}\qquad \text{and}\qquad
   \s_n^i(s_j)=\begin{cases} s_j&j<i\\
      *&j=i\\
      s_{j-1}&j>i.\end{cases}
  \end{equation}
  It remains to note that any morphism $\p$ is determined by the
  images $\p(s_j)$.
\end{proof}

We denote by $\Si$ the \emph{category of connected strings} with
objects given by the strings $\Si_n$, $n\in\bbN$.

Any morphism can be written in some canonical form. First observe that
there is a canonical epi-mono factorisation.

\begin{lem}\label{le:epi-mono}
  Let $\p\colon\Si_m\to\Si_n$ be a morphisms. Let
  $0\le j_v<\cdots < j_0< m$ be the indices $j$ such that
  $\p(s_{j})=*$ and set $\p'=\s_{m-v}^{j_v}\circ\cdots
  \circ\s_{m-1}^{j_1} \circ\s_{m}^{j_0}$. Then there is a
  factorisation $\p=\p''\circ\p'$ such that $\p''$ is injective.
\end{lem}
\begin{proof}
  When $\p(s_i)=*$ for some simple string $s_i$, then
  $\p(s_is)=\p(s)=\p(ss_i)$ for all $s\in\Si_m$. This yields the
  factorisation $\p=\p''\circ\p'$.  Given strings $s,t$ such that
  $\p''(s)=\p''(t)$, an induction on their length shows that $s=t$.
\end{proof}
  
\begin{lem}\label{le:decomp}
  Every morphism $\p\colon\Si_m\to\Si_n$ can be written uniquely as
  composite
  \begin{equation}\label{eq:decomp}
    \p=\d_n^{i_u}\circ\d_{n-1}^{i_{u-1}}\circ\cdots \circ
    \d_{n-u}^{i_{0}}\circ\s_{m-v}^{j_v}\circ\cdots \circ\s_{m-1}^{j_1} \circ\s_{m}^{j_0}
  \end{equation}
  with $0\le i_0<\cdots< i_u\le n$, $0\le j_v<\cdots< j_0< m$, and
  $n-u=m-v$.\footnote{We need to exclude $\d_1^0$ as a factor and
    choose instead $\d_1^1$ in order to achieve uniqueness.}
\end{lem}
We call  \eqref{eq:decomp} the \emph{canonical decomposition} of $\p$
in $\Si$.
\begin{proof}
  Let $0\le j_v<\cdots< j_0< m$ be the indices $j$ such that
  $\p(s_{j})=*$. And let $0\le i_0<\cdots< i_u\le n$ be the
  indices such that for all $(s',s'')\in\Si_n$ in the image of $\p$ we
  have $s',s''+1\not\in\{i_0,\ldots,i_u\}$. Then it is clear that $\p$
  satisfies \eqref{eq:decomp}. Conversely, if $\p$ is written as in
  \eqref{eq:decomp}, then the indices $i_0,\cdots, i_u$ and
  $j_0,\cdots, j_v$ are characterised as above.
\end{proof}

\begin{lem}\label{le:cat-gen}
  The category $\Si$ is generated by the objects $\Si_n$, the
  morphisms $\d_n^i$, $\s_m^j$, plus the relations $\d_1^0=\d_1^1$ and
  \eqref{eq:rel}.
\end{lem}
\begin{proof}
  Let us denote by $\Si'$ the category generated by the objects
  $\Si_n$, the morphisms $\d_n^i$, $\s_m^j$, plus the relations
  $\d_1^0=\d_1^1$ and \eqref{eq:rel}. Since the relations are
  satisfied in $\Si$, there is a unique functor $\Si'\to\Si$ which
  induces the identity on the objects and on the morphisms $\d_n^i$
  and $\s_m^j$. Since every morphism in $\Si$ is a composite of
  morphisms $\d_n^i$ and $\s_m^j$, the induced map
  $\Hom_{\Si'}(\Si_m,\Si_n)\to \Hom_{\Si}(\Si_m,\Si_n)$ is surjective
  for all $m,n$. To show injectivity, choose $\p,\psi$ in
  $\Hom_{\Si'}(\Si_m,\Si_n)$ with same image in $\Si$. Since the
  relations  $\d_1^0=\d_1^1$ and \eqref{eq:rel} in $\Si'$ are satisfied, there are
  decompositions \eqref{eq:decomp} in $\Si'$ for $\p$ and
  $\psi$. These decompositions coincide in $\Si$ by
  Lemma~\ref{le:decomp}, and therefore $\p=\psi$.
\end{proof}

\begin{exm}
For any $n\ge 3$  we have the following pullback in $\Si$.
\[\begin{tikzcd}
    \Si_n\arrow{r}{\s_n^0}\arrow{d}{\s_n^{n-1}}&\Si_{n-1}\arrow{d}{\s_{n-1}^{n-2}}\\
     \Si_{n-1}\arrow{r}{\s_{n-1}^0}&\Si_{n-2}
  \end{tikzcd}\]
\end{exm}

\section{The simplicial category}

Let $\De$ denote the \emph{simplicial category}, which is also known
as \emph{augmented simplex category} (terminology and notation follows
\cite[VII.5]{ML1998}). The objects are given by the finite ordinals
$[n]=\{0,1, \ldots,n-1\}$, $n\in\bbN$, and the morphisms
$\p\colon [m]\to [n]$ are given by maps satisfying $\p(i)\le \p(j)$ for
all $0\le i\le j<m$. For $n\ge 0$ there is the \emph{face map}
 \[\bar\d^i_n\colon [n]\lto [n+1] \qquad (0\le i\le n)\] 
 (the unique injective map not taking the value $i$) and for $n\ge 1$
 the \emph{degeneracy map}
\[\bar\s^i_n\colon [n+1]\lto [n] \qquad (0\le i< n)\]
(the unique surjective map taking twice the value $i$) which are
 known to satisfy the simplicial identities \eqref{eq:rel}. In fact, the category
 $\De$ is generated by the objects $[n]$, the morphisms $\bar\d_n^i$,
 $\bar\s_m^j$, and the identities \eqref{eq:rel}; see \cite[II.2]{GZ1967}
or \cite[VII.5]{ML1998}.

We write $\bar\De=\De[\a^{-1}]$ for the category which is obtained
by formally inverting the morphism $\a=\bar\d_0^0$. This amounts to identifying  the initial and
the terminal object in $\De$.

\begin{prop}\label{pr:simplex}
  The assignments 
  \[[n]\mapsto \begin{cases} \Si_0&n=0\\\Si_{n-1}&n>0,\end{cases} \qquad \bar\d_{n}^i\mapsto \d_{n}^i,\qquad \bar\d_0^0\mapsto\id,\qquad \bar\s_n^i\mapsto
    \s_{n}^i\] provide a functor $p\colon \De\to\Si$ which induces an
  equivalence $\bar\De\iso\Si$.
\end{prop}
\begin{proof}
  The functor $p$ is well defined since it maps
  generators to generators and the simplicial identities are satisfied
  in both categories. The functor $p$ inverts $\bar\d_0^0$ and induces
  therefore a functor $\bar\De\to\Si$, which yields a bijection
  between the isomorphism classes of objects. Also for the morphisms
  we obtain bijections since the functor matches generators and
  relations. Note that $\bar\d_1^0=\bar\d_1^1$ in $\bar\De$ since
  $\bar\d_0^0$ is invertible.
\end{proof}

The functor  $p\colon \De\to\Si$ admits two sections $s_0$ and $s_1$
that are given by
\[ \Si_n\mapsto [n+1], \qquad \d_{n}^i\mapsto \bar\d_{n}^i,\qquad
  \s_n^i\mapsto \bar\s_{n}^i,\] except that
$s_i(\d_1^0)=\bar\d_1^i =s_i(\d_1^1)$ for $i=0,1$.

Let us summarise. We have for all $n\ge 1$ diagrams 
\[
  \begin{tikzcd}
\Si_{n-1} \arrow[rr,yshift=-1.5ex,"\d_n^{i}",swap]
  \arrow[rr,yshift=1.5ex,"\d_n^{i+1}"]
 &&\Si_n \arrow{ll}[description]{\s_n^i} 
\end{tikzcd} \qquad (0\le i < n)
\]
satisfying the simplicial identities, but a difference from the usual
simplex category arises because of the extra identity
$\d_1^0=\d_1^1$.

The equivalence in Proposition~\ref{pr:simplex} can be explained in
terms of linear representations of posets. We refer to
Theorem~\ref{th:abelian} and the appendix for further details.
  
\section{Representations}

The category of finite strings models certain categories of linear
representations. In the following we specify the relevant class of
abelian categories and the exact functors between them.

Let $P$ be a poset and $k$ a field. A $k$-linear
representation of $P$ is by definition a functor $P\to\mod k$ into the
category of finite dimensional $k$-spaces, where $P$ is viewed as a category (with
objects the elements in $P$ and a unique morphism $x\to y$ iff
$x\le y$). Morphisms between representations are the natural
transformations, and we denote by $\Rep(P,k)$ the category of all
finite dimensional $k$-linear representations.

For a $k$-linear abelian category $\A$ over a field $k$ we consider
the following conditions.
  \begin{enumerate}
  \item[(Ab1)] $\A$ is \emph{connected}, that is, $\A=\A_1\times\A_2$ implies $\A_1=0$ or $\A_2=0$.
  \item[(Ab2)] $\A$ is a \emph{length category}, that is, every object has a finite composition
    series, and  there are only finitely many isomorphism classes of simple objects.
  \item[(Ab3)] $\A$ is \emph{hereditary}, that is, $\Ext^2$ vanishes.
  \item[(Ab4)] $\A$ is \emph{uniserial}, that is, every indecomposable object has a unique
    composition series.
  \item[(Ab5)] $\A$ is \emph{split}, that is, $\End(S)\cong k$ for every simple object $S$.
  \item[(Ab6)] $\A$ is of \emph{finite type}, that is, there are only finitely many
    isomorphism classes of indecomposable objects.
  \end{enumerate}

\begin{lem}\label{le:uniserial}
  Let $\A$ be a $k$-linear abelian category satisfying
  \emph{(Ab1)--(Ab6)}.  Then there is an equivalence
  $\A\iso \Rep([n]^\op,k)$, where $n$ equals the number of isomorphism
  classes of simple objects in $\A$.
\end{lem}
\begin{proof}
  See for example the description of uniserial categories in
  \cite{AR1968}.
\end{proof}

From now on we fix a field $k$ and set $\A_n:=\Rep([n]^\op,k)$ for
$n\in\bbN$.  An object $M\in\A_n$ is a diagram
\[M(n-1)\lto\cdots\lto M(1)\lto M(0)\]
of $k$-spaces. For any string $s\in\Si_n$ we define a representation
$M_s\in\A_n$ as follows. Set $M_*=0$.  For $s=(s',s'')$ let $M_s$ be
the representation\footnote{This is also known as \emph{string module}
in the terminology of \cite{BR1987}.}
\[0\lto\cdots \lto 0\lto k\xto{\ 1\ }\cdots\xto{\ 1\ }k\lto 0\lto\cdots\lto 0\] such that
$M_s(i)=k$ iff $s'\le i\le s''$. Set $M_i:=M_{s_i}$ for
$0\le i <n$. Observe that the composition length of $M_s$ equals
$\ell(s)$, and the composition factors of $M_s$ correspond bijectively
to the composition factors of $s$.

\begin{lem}\label{le:lin-alg}
  Let  $n\in\bbN$.
  \begin{enumerate}
  \item  The assignment $s\mapsto M_s$ induces a bijection
  between $\Si_n\setminus\{*\}$ and the isomorphism classes of
  indecomposable objects in  $\A_n$.
  \item For $s,t,u\in\Si_n$ there is an exact sequence $0\to M_s\to M_t\to M_u\to 0$ if
    and only if $t=su$.
  \end{enumerate}
\end{lem}
\begin{proof}
  Straightforward.
\end{proof}

Next we specify the class of exact functors which arises naturally in
our context. For each exact functor $F\colon\A\to\B$ between abelian
categories we denote by $\Ker F$ the full subcategory of $\A$ given by
the objects $X\in\A$ such that $FX=0$. This is a Serre subcategory and
we denote by $\A/(\Ker F)$ the corresponding quotient, cf.\
\cite{Ga1962}.

We say that an exact functor $F\colon\A\to\B$ between abelian
categories admits a \emph{homological factorisation} if the induced
functor $\A/\A'\to \B$ with $\A'=\Ker F$ induces for all objects
$X,Y\in\A$ bijections
  \[\Ext_{\A/\A'}^i(X,Y)\lto \Ext_\B^i(FX,FY) \qquad (i\ge 0).\]

A full subcategory of an abelian category
is \emph{thick} if it is closed under direct summands and the
\emph{two out of three property} holds for any short exact sequence
(that is, if two terms belong to the subcategory, then also the
third).

\begin{lem}
An exact functor  $F\colon\A\to\B$  between hereditary abelian
  categories admits a \emph{homological factorisation} 
 if and only if $\A/(\Ker F)$ identifies with a thick subcategory of $\B$.
\end{lem}
\begin{proof}
  Set $\A'=\Ker F$ and suppose $F$ identifies $\A/\A'$ with a full
  subcategory $\B'\subseteq\B$.  Clearly, $\B'$ is closed under
  kernels and cokernels of morphisms since $F$ is exact. The
  subcategory $\B'$ is extension closed if and only if the induced map
  $\Ext_{\A/\A'}^1(X,Y)\to \Ext_\B^1(FX,FY)$ is a bijection for all
  $X,Y\in\A$.
\end{proof}  

Not all exact functors admit a homological factorisation.  A simple
example is for any field $k$ the exact functor $\mod k\to\mod k$ given
by $X\mapsto X\otimes_k k^2$.

For $m,n\in\bbN$ we denote by $\Hom(\A_m,\A_n)$ the set of $k$-linear
exact functors $\A_m\to\A_n$, up to natural isomorphism, that
admit a homological factorisation.  We define natural maps
\[\Hom(\A_m,\A_n)\xto{\ \a_{mn}\ } \Hom(\Si_m,\Si_n)\quad\text{and}\quad
  \Hom(\Si_m,\Si_n)\xto{\ \b_{mn}\ } \Hom(\A_m,\A_n)\] as follows.

Any morphism $\p\colon [m]\to [n]$ induces an exact functor
$\p^*\colon\A_n\to \A_m$ via precomposition. Let us set
$s_n^i:=(\bar\d_{n-1}^{i})^*$ for $0\le i< n$.
  
\begin{lem}\label{le:recollement}
 Let $n\ge 1$. There are canonical recollements of abelian
categories
  \begin{equation*}
\begin{tikzcd}
  \A_1 \arrow[tail]{rr} &&\A_n
  \arrow[twoheadrightarrow,yshift=-1.5ex]{ll}
  \arrow[twoheadrightarrow,yshift=1.5ex]{ll}
  \arrow[twoheadrightarrow]{rr}[description]{s_n^i} &&\A_{n-1}
  \arrow[tail,yshift=-1.5ex]{ll}{d_n^{i}}
  \arrow[tail,yshift=1.5ex]{ll}[swap]{d_n^{i+1}}&(0\le i < n)
\end{tikzcd}
\end{equation*}
such that
\[\Ker s_n^i=\add M_i,\qquad \Im d_n^{i}={M_i}^\perp, \qquad \Im
  d_n^{i+1}={^\perp M_i}.\]
The functors $s_n^i$, $d_n^{i}$, $d_n^{i+1}$ are exact. Moreover, they
send indecomposable objects to indecomposable objects or to zero.
\end{lem}
\begin{proof}
  Let $\C=\Ker s_n^i$ denote the full subcategory of objects in $\A_n$
  that are annihilated by $s_n^i$.  It is clear that $M_i$ is the
  unique simple object in $\C$. Thus $\C$ equals the full subcategory
  given by the finite direct sums of copies of $M_i$. Then the right
  adjoint of the quotient functor $\A_n\to\A_n/\C$ identifies
  $\A_n/\C$ with $\C ^\perp$, while the left adjoint identifies
  $\A_n/\C$ with $^\perp\C$. Here, we consider the perpendicular
  categories defined with respect to $\Hom$ and $\Ext^1$; cf.\
  \cite[III.2]{Ga1962}.  This yields the descriptions of $d_n^{i}$ and
  $d_n^{i+1}$. The embedding of any perpendicular category into $\A_n$
  is exact since $\Ext^2$ vanishes. The functor $s_n^i$ annihilates
  $M_i$ and sends all other indecomposables to indecomposable objects.
\end{proof}

For $n\ge 1$ we set
\[\b_{n,n-1}(\s_n^i):= s_n^i\qquad\text{and}\qquad\b_{n-1,n}(\d_n^{i}):=d_n^{i}.\]
One checks that these functors satisfy the identities \eqref{eq:rel}.
Thus the assignment extends uniquely to maps
$\b_{mn}\colon \Hom(\Si_m,\Si_n)\to \Hom(\A_m,\A_n)$ for all
$m,n\in\bbN$, using Lemma~\ref{le:cat-gen}. 

\begin{rem}
We have $d_n^i=(\bar\s_{n-1}^{i-1})^*$ for $0< i< n$. Thus $d^0_n$
and $d^n_n$ are not obtained from morphisms $[n-1]\to [n]$.
\end{rem}

\begin{lem}\label{le:exact-fun}
  Let $m,n\in\bbN$.  An exact functor $F\colon \A_m\to \A_n$ that
  admits a homological factorisation
  induces a morphism $\p\colon \Si_m\to\Si_n$ which is given by
  $F(M_s)=M_{\p(s)}$.
\end{lem}
\begin{proof}
  The functor $F$ identifies $\A_m/(\Ker F)$ with a full subcategory
  of $\A_n$.  The canonical functor $\A_m\to \A_m/(\Ker F)$ can be
  written as composite of functors of the form
  $s_p^i\colon\A_p\to\A_{p-1}$, which map indecomposables either to
  indecomposables or to zero.  Thus for any $s\in\Si_m$ we have
  $F(M_s)= M_t$ for some $t\in\Si_n$, using Lemma~\ref{le:lin-alg}.
  This yields a morphism $\p\colon \Si_m\to\Si_n$ by setting
  $\p(s)=t$.
\end{proof}

The above lemma provides maps $\a_{mn}\colon \Hom(\A_m,\A_n)\to
\Hom(\Si_m,\Si_n)$ satisfying  $\a_{nn}(\id)=\id$ and $\a_{mp}(G\circ
F)=\a_{np}(G)\circ \a_{mn}(F)$ for any pair of composable functors
$F,G$.

\begin{lem}\label{le:standard}
  A $k$-linear equivalence $\A_n\iso\A_n$ is naturally isomorphic to
  the identity.
\end{lem}
\begin{proof}
  The category $\A_n$ is \emph{standard}, that is, equivalent to the
  mesh category given by its Auslander-Reiten quiver
  \cite[2.4]{Ri1984}. Clearly, an equivalence induces the identity on the
  Auslander-Reiten quiver and preserves the mesh ideal. From this the
  assertion follows.
\end{proof}

\begin{lem}\label{le:morphisms-linear}
  Let $m,n\in\bbN$. Then $\b_{mn}\circ\a_{mn}=\id$ and
  $\a_{mn}\circ\b_{mn}=\id$.
\end{lem}
\begin{proof}
  The identity $\a_{mn}\circ\b_{mn}=\id$ is clear since this can be
  checked on the standard morphisms, thanks to Lemma~\ref{le:cat-gen}.
  We consider only $k$-linear exact functors $F\colon \A_m\to \A_n$
  that admit a homological factorisation. Such functors are
  determined, up to natural isomorphism, by the values $F(M_s)$ of the
  indecomposable objects; see Lemma~\ref{le:standard}. Thus $\a_{mn}$ is injective and
  $\b_{mn}\circ\a_{mn}=\id$ follows.
\end{proof}
  
Combining the above lemmas yields a combinatorial description of the
abelian categories that are specified in Lemma~\ref{le:uniserial}.

\begin{thm}\label{th:abelian}
Let $k$ be a field. The assignment $\Si_n\mapsto\A_n$
provides an equivalence between the category of
connected strings and the category of $k$-linear abelian  categories
satisfying \emph{(Ab1)--(Ab6)}.\qed
\end{thm}

We refer to the appendix for some further explanation of this result.

\section{Finite coproducts}

For a finite set of natural numbers $n_\a\in \bbN$ we define the
coproduct $\coprod_{\a}\Si_{n_\a}$ of strings by taking from the product
of the underlying sets all elements $s=(s_\a)$ such that $s_\a\neq *$
for at most one index $\a$ (that is, the coproduct of the pointed sets
$\Si_{n_\a}$). For $s=(s_\a)$ and $t=(t_\a)$ in
$\coprod_{\a}\Si_{n_\a}$ set
\[st:=(s_\a t_\a).\] For each index $\a$ and $0\le i<n_{\a}$ we denote
by $s_{\a,i}$ the \emph{simple string} $s$ given by $s_{\a}=s_i$.

Each coproduct $\coprod_{\a}\Si_{n_\a}$ comes with canonical
inclusions $i_\a\colon \Si_{n_\a}\to \coprod_{\a}\Si_{n_\a}$ and
projections $p_\a\colon \coprod_{\a}\Si_{n_\a}\to \Si_{n_\a}$
satisfying $p_\a\circ i_\a =\id$.

Morphisms $\coprod_{\a}\Si_{m_\a}\to \coprod_{\b}\Si_{n_\b}$ are by
definition maps $\p$ between the underlying sets such that the
composite $p_\b\circ\p\circ i_\a$ is a morphism
$\Si_{m_\a}\to\Si_{n_\b}$ for all $\a,\b$.

\begin{lem}
  There are canonical isomorphisms of pointed sets
\[\Hom\Big(\coprod_{\a}\Si_{m_\a},
  \coprod_{\b}\Si_{n_\b}\Big)\iso \prod_{\a}\Hom\Big(\Si_{m_\a},
  \coprod_{\b}\Si_{n_\b}\Big)\leftiso
  \prod_{\a}\coprod_{\b}\Hom(\Si_{m_\a}, \Si_{n_\b}).\]
\end{lem}

\begin{proof}
  Isomorphisms of pointed sets are nothing but bijections, but it is
  important to take (co)products of pointed sets.  The first bijection
  is induced by the canonical inclusions
  $\Si_{m_\a}\to \coprod_{\a}\Si_{m_\a}$. The second bijection uses
  the fact that each morphism $\Si_{m_\a}\to \coprod_{\b}\Si_{n_\b}$
  factors through the inclusion $\Si_{n_\b}\to\coprod_{\b}\Si_{n_\b}$
  for one index $\b$.
\end{proof} 

We obtain the \emph{category of finite strings} which has as objects
the finite coproducts of connected strings.\footnote{We may consider
  the category $\HOM(\Si^\op,\Set_*)$ of functors $\Si^\op\to\Set_*$
  into the category of pointed sets, which is the analogue of the
  category $\HOM(\De^\op,\Set)$ of simplicial sets. Then the category of
  finite strings identifies via the embedding $X\mapsto\Hom(-,X)|_\Si$
  with the full subcategory of finite coproducts of representable
  functors in $\HOM(\Si^\op,\Set_*)$.}

\section{Non-crossing partitions}

We wish to describe the subobjects of $\Si_n$ in the category of
finite strings. This requires some preparations.

Let $S\subseteq\Si_n$. We call $S$ \emph{thick} if $*\in S$ and for any
pair $s,t\in S$ of non-zero strings we have $st\in S$, and moreover
$(s', t'-1), (t',s''), (s''+1,t'')\in S$ provided that
$s'\le t'\le s''\le t''$. We denote by $\Thick(S)$ the smallest thick
subset of $\Si_n$ containing $S$.

A set $S\subseteq \Si_{n}$ of non-zero strings
is called \emph{non-crossing} provided that $s,t\in S$ and
$s'\le t'\le s''\le t''$ implies $s=t$.

\begin{lem}
The assignment $S\mapsto\Thick(S)$ gives a bijection between the
non-crossing subsets and the thick subsets of $\Si_n$.
\end{lem}
\begin{proof}
  The inverse map takes a thick subset $T\subseteq\Si_n$ to the unique
  non-crossing subset $S\subseteq T$ with $\Thick(S)=T$.
\end{proof}

For non-crossing subsets $S,S'$ of $\Si_n$ we set
\[S\le S'\quad :\iff\quad \Thick(S)\subseteq\Thick(S').\] This yields
the structure of a poset. In fact, the non-crossing subsets form a
lattice since the thick subsets of $\Si_n$ are closed under
intersections. We denote this lattice by $\NC(\Si_n)$.

Let $n\in\bbN$. A \emph{partition} $P=(P_\a)$ of $[n]$ is given by
pairwise disjoint non-empty subsets $P_\a$ of $[n]$ such that
$\bigcup_\a P_\a=[n]$. Each partition is determined by the
corresponding set of strings $S(P)\subseteq\Si_{n-1}$, where by
definition $s=(s',s'')\in S(P)$ if for some $\a$ we have
$s',s''\in P_\a$ and $i\not\in P_\a$ for all $s'<i\le s''$. This is
clear since any part $P_\a=\{a_1<a_2<\cdots<a_r\}$ is determined by
the corresponding set of strings
$S_\a=\{(a_1,a_2-1),\ldots,(a_{r-1},a_r-1)\}$.

Call a subset $S\subseteq\Si_{n-1}$ of non-zero strings
\emph{partitioning} when for any $s,t\in S$ we have $s'=t'$ iff
$s''=t''$. In that case there is a unique partition $P=P(S)$ such that
$S(P)=S$. This yields a bijective correspondence between partitions
of $[n]$ and partitioning sets of strings in $\Si_{n-1}$.

A partition $P$ is \emph{non-crossing} provided given elements
$i<j < i' < j'$ with $i, i'$ in the same part and $j, j'$ in the same
part, then all elements belong to the same part.
The partitions of $[n]$ are partially ordered via \emph{refinement}, so
$P\le P'$ if any part of $P$ is contained in a part of $P'$. The
non-crossing partitions then form a lattice which is denoted by
$\NC(n)$; cf.\ \cite{Kr1972,Si2000}.

\begin{lem}\label{le:nc}
 There is a lattice isomorphism
  $\NC(\Si_{n-1})\iso\NC(n)$ which is given by  $S\mapsto P(S)$.
\end{lem}
\begin{proof}
  It is clear that $S\subseteq\Si_{n-1}$ is non-crossing if and only
  if $P(S)$ is non-crossing. Let $S\le S'$. This means any $s\in S$ can
  be written as $s=s_1s_2\cdots s_r$ with  $s_1,\ldots,s_r$ in $S'$. On the
  other hand, $P\le P'$ means that for any part
  $P_\a=\{a_1<a_2<\cdots<a_u\}$ of $P$ and $t=(a_i,a_{i+1}-1)\in S(P)$, there is
  a part of $P'$ containing $a_i,a_{i+1}$ and therefore $t=t_1t_2\cdots t_r$
  with $t_1,\ldots,t_r$ in $S(P')$. Thus $S\le S'$ if and only if $P(S)\le P(S')$.
\end{proof}

We say that two monomorphisms
$X_1\rightarrowtail X$ and $X_2\rightarrowtail X$ are
\emph{equivalent}
if there exists an isomorphism $X_1\to X_2$ making the following
diagram commutative.
\begin{equation*}
\begin{tikzcd}[row sep=scriptsize, column sep=scriptsize]
X_1\arrow[tail,rd]\arrow[rr]&&X_2\arrow[tail,ld]\\&X
\end{tikzcd}
\end{equation*}
An equivalence class of monomorphisms into $X$ is called a
\emph{subobject} of $X$. Given subobjects $X_1\rightarrowtail X$
and $X_2\rightarrowtail X$, we write $X_1\le X_2$ if there is a
morphism $X_1\to X_2$ making the above diagram commutative; this
yields a partial order.

For a monomorphism $\p\colon X\to\Si_n$ in the category of finite
strings we set 
\[S(\p):=\{\p(s)\mid s\in X \text{ simple}\}.\]

\begin{lem}\label{le:exist}
  Let $S\subseteq\Si_n$ be non-crossing. Then there exists a
  monomorphism $\p\colon X\to\Si_n$ such that $S(\p)=S$.
\end{lem}
\begin{proof}
  Consider the equivalence relation on $S$ generated by $s\sim t$ when
  $st\neq\ast$. This yields a partition $S=\bigcup_\a S_\a$ and we set
  $n_\a:=\card S_\a$. Using the fact that $S$ is non-crossing, there
  is a unique morphism $\p\colon \coprod_{\a}\Si_{n_\a}\to\Si_n$ which
  identifies the simple strings $s_{\a,i}$ with the elements in $S_\a$. Thus
  $S=S(\p)$.
\end{proof}

\begin{lem}\label{le:mono-inj}
  A morphism $\p\colon X\to\Si_n$ is a monomorphism if and only if it
  is given by an injective map.
\end{lem}
\begin{proof}
  Clearly, any injective map yields a monomorphism. Thus we suppose
  that $\p$ is a monomorphism and need to show that $\p$ is given by
  an injective map.
 
  Let $X=\coprod_{i=1}^r\Si_{n_i}$. The canonical decomposition of a
  morphism $\Si_{n_i}\to\Si_n$ from Lemma~\ref{le:decomp} yields the
  case $r=1$.  For the general case we may assume that $r=2$. For each
  index $i$ the restricted morphism $\p_i\colon\Si_{n_i}\to \Si_n$ is
  given by an injective map by the first case. Then each subset
  $\Im\p_i$ is thick, and $\Im\p_1\cap\Im\p_2=\Thick(S)$ for some
  non-crossing $S\subseteq\Si_n$. Let $\psi\colon Y\to \Si_n$ be the
  corresponding morphism with $S(\psi)=S$ which exists by
  Lemma~\ref{le:exist}. Clearly, $\psi$ factors through each $\p_i$
  via a morphism $\psi_i\colon Y\to \Si_{n_i}$. We obtain a diagram
\[\begin{tikzcd}[column sep=large]
 Y\arrow[yshift=.75ex]{r}{\psi_1}\arrow[swap,yshift=-.75ex]{r}{\psi_2}&X\arrow{r}{\p}&\Si_n
\end{tikzcd}\] where both composites equal $\psi$.  Thus
$\psi_1=\psi_2$ and therefore $\Im\p_1\cap\Im\p_2=\{*\}$.  We conclude
that $\p$ is given by an injective map.
\end{proof}

\begin{lem}\label{le:mono}
  Let $\p\colon X\to\Si_n$ be a monomorphism.  Then the set $S(\p)$ is
  non-crossing and we have $\Thick(S(\p))=\Im\p$. Moreover, $\p$
  factors through a monomorphism $\p'\colon X'\to\Si_n$ if and only if
  $S(\p)\le S(\p')$.
\end{lem}
\begin{proof}
 Let $X=\coprod_{i=1}^r\Si_{n_i}$ and $m\in\bbN$.  The set $S$ of simple
  strings in $\Si_m$ is non-crossing and we have
  $\Thick(S)=\Si_m$. This property is preserved under a monomorphism
  $\Si_m\to\Si_n$ and yields the case $r=1$.  The general case follows
  since  the  restrictions $\p_i\colon\Si_{n_i}\to \Si_n$  satisfy
  $\Im\p_i\cap\Im\p_j=\{*\}$ for $i\neq j$, by Lemma~\ref{le:mono-inj}.

  For a monomorphism $\p'\colon X'\to\Si_n$ we have
  \begin{align*}
    S(\p)\le S(\p') &\iff \Thick(S(\p))\subseteq \Thick(S(\p'))\\
                    & \iff \Im\p\subseteq\Im\p'\\
    &\iff \p\text{ factors through }\p'.\qedhere
    \end{align*}
  \end{proof}

\begin{thm}\label{th:sub}
  Let $n\in\bbN$. The subobjects of $\Si_n$ in the category of finite
  strings form a lattice which is canonically isomorphic to the
  lattice of non-crossing partitions $\NC(n+1)$. The isomorphism sends
  a monomorphism $\p\colon X\to\Si_n$ to $P(S(\p))$.
\end{thm}

This result could be deduced from \cite{HK2016,IT2009}, using the
correspondence between strings and representations from
Theorem~\ref{th:abelian}, which identifies thick subsets of
$\Si_n$ with thick subcategories of $\A_n$. We refer to
\cite[\S4]{Ri2016} for a detailed exposition. The following is a
direct proof.
  
\begin{proof}
  The assignment $\p\mapsto P(S(\p))$ gives a well defined map from the
  poset of subobjects of $\Si_n$ to $\NC(n+1)$ by Lemmas~\ref{le:nc}
  and \ref{le:mono}. In fact, the map is injective and $\p$ factors
  through a monomorphism $\p'$ if and only if $P(S(\p))\le P(S(\p'))$.  Thus it
  remains to show surjectivity. Let $P\in\NC(n+1)$ and set
  $S=S(P)$. Then there is a morphism
  $\p\colon \coprod_{\a}\Si_{n_\a}\to\Si_n$ satisfying $S=S(\p)$ by
  Lemma~\ref{le:exist}. Thus $P=P(S(\p))$.
\end{proof}

\section{Cyclic strings}

We enlarge the category of finite strings and add cyclic strings as
follows.  Let $\Si_\bbZ$ denote the set of all basic strings. There is
a natural action of the group of integers given by
\[*^z:=* \qquad\text{and}\qquad s^z:=(s'+z,s''+z)\qquad \text{for }s=(s',s''),\, z\in\bbZ.\]
For $n>0$ the \emph{cyclic string} of length $n$ is the set of orbits
with respect to the action of the subgroup $(n)=n\bbZ$. Thus
\[\tilde\Si_n:=\{s^{(n)}\mid s\in\Si_\bbZ\}=\{s^{ni}\mid
  s\in\Si_\bbZ,\, i\in\bbZ\}\]
with multiplication given by \[s^{(n)}t^{(n)}:=u^{(n)}\]
where $u=*$ except when there is a pair of integers $i,j$ such that
$s^{ni}t^{nj}=u\neq *$. We set $\tilde \Si_0:= \{*\}$.

A morphism $\p\colon\tilde\Si_m\to\tilde\Si_n$ is by definition a map
satisfying \eqref{eq:mor}. We define \emph{standard morphisms} which
are given by their values on simple strings as in \eqref{eq:standard}.
Let $n\ge 1$. Then the morphism
  \[\tilde\d^i_n\colon \tilde\Si_{n-1}\lto\tilde\Si_n \qquad (0\le i < n)\] 
  is given by the injective map such that $s^{(n)}_{i-1}$ and
  $s^{(n)}_{i}$ are not in its image, and the
  morphism
  \[\tilde\s^i_n\colon \tilde\Si_{n}\lto\tilde\Si_{n-1} \qquad (0\le i
    <n)\] is given by the surjective map such that $s^{(n)}_i$ is sent
  to the zero string. The \emph{cyclic
    permutation}
  \[\t^i_n\colon\tilde\Si_n\lto\tilde\Si_n \qquad (0\le i
    <n)\]  is given by
  $s^{(n)}\mapsto (s^i)^{(n)}$.

\begin{lem}\label{le:decomp-cyc}
  The standard morphisms satisfy the identities \eqref{eq:rel}, and
  every morphism $\p\colon\tilde\Si_m\to\tilde\Si_n$ admits a unique
  decomposition
\begin{equation*}\label{eq:decomp-cyc}
    \p=\d_n^{i_u}\circ\d_{n-1}^{i_{u-1}}\circ\cdots \circ
    \d_{n-u}^{i_{0}}\circ\s_{m-v}^{j_v}\circ\cdots \circ\s_{m-1}^{j_1} \circ\s_{m}^{j_0}\circ\t_m^k
  \end{equation*}
 with $0\le i_0<\cdots< i_u< n$, $0\le j_v<\cdots< j_0< m$, $0\le
 k<m$, and
 $n-u=m-v$.
\end{lem}
\begin{proof}
  Adapt the proof of Lemmas~\ref{le:rel} and \ref{le:decomp}. The only
  difference arises from cyclic permutations.
\end{proof}

Next we consider morphisms $\Si_m\to\tilde\Si_n$ and
$\tilde\Si_m\to\Si_n$, which are by definition maps satisfying
\eqref{eq:mor}.  Let $n\ge 1$. The standard morphism
  \[\e^i_n\colon \Si_{n-1}\lto\tilde\Si_n \qquad (0\le i < n)\] 
  is given by the unique injective map such that $s^{(n)}_{i}$
  is not in its image. We consider the morphism
  \[\begin{tikzcd}
      \tilde\Si_n\arrow{r}{\t_n^{-i}}&\tilde\Si_n\arrow{r}{\tilde\s_n^{n-1}}&
      \tilde\Si_{n-1}\arrow{r}{\tilde\s_{n-1}^{n-2}}&\cdots\arrow{r}{\tilde\s_2^{1}}&\tilde\Si_1
    \end{tikzcd}\] and note that its \emph{kernel} (that is, the set
  of elements sent to the zero string) equals the image of $\e_n^i$.
  
\begin{lem}\label{le:factorisation}
  Let $m,n\in\bbN$. Every morphism $\Si_m\to\tilde\Si_n$ factors
  through $\e^i_n$ for some $0\le i < n$, and every morphism
  $\tilde\Si_m\to\Si_n$ factors through $\tilde\Si_0=\Si_0$.
\end{lem}
\begin{proof}
  First consider a morphism $\p\colon\Si_m\to\tilde\Si_n$. It is
  easily checked that the longest string $s^{(n)}$ in the image of
  $\p$ has length at most $n-1$, because the image of $\p$ is
  finite. Choose an index $i$ such that the simple $s_i$ does not
  arise as a composition factor of $s$. It follows that $\p$ factors
  through $\e_n^i$, since all composition factors of strings in the
  image of $\p$ are composition factors of $s$.

  Now consider for $\psi\colon\tilde\Si_m\to\Si_n$ its epi-mono
  factorisation $\psi=\psi''\circ\psi'$. The image is of the form
  $\tilde\Si_p$ for some $p\le m$. Because $\psi''$ is injective and
  $\tilde\Si_p$ is infinite for $p>0$ we conclude that $p=0$.
\end{proof}

Let us consider the category of all connected strings (linear and cyclic). The objects are
of the form $\Si_n$ or $\tilde\Si_n$ with $n\in\bbN$. As before, we
add finite coproducts and obtain the \emph{enlarged category of finite
  strings}. The objects are of the form
\[\Big(\coprod_{\a}\Si_{m_\a}\Big)\amalg
  \Big(\coprod_{\b}\tilde\Si_{n_\b}\Big)\]
given by a finite set of natural numbers $m_\a$ and $n_\b$.

Let $k$ be a field. For the quiver
\[\begin{tikzcd}[ampersand replacement=\&]
   % [\tilde n]\colon
    n-1\arrow{r}\&n-2\arrow{r} \&\cdots
\arrow{r}\& 1\arrow{r}\&0 \arrow[bend
left=12]{llll}\&(n\ge 1)
\end{tikzcd}\] we denote by $\tilde\A_n$ the category of all finite
dimensional and nilpotent $k$-linear representations. Then we have the following
analogue of Lemma~\ref{le:uniserial}.

\begin{lem}\label{le:uniserial-cyc}
  Let $k$ be a field and $\A$ a $k$-linear abelian category. Suppose
  that $\A$ satisfies \emph{(Ab1)--(Ab5)} but not \emph{(Ab6)}.
   Then there is an equivalence $\A\iso \tilde\A_n$, where $n$
  equals the number of isomorphism classes of simple objects in $\A$.
\end{lem}
\begin{proof}
  See for example the description of uniserial categories in
  \cite{AR1968}.
\end{proof}

We continue with analogues of Lemmas~\ref{le:lin-alg} and
\ref{le:recollement}. Let $n\ge 1$. The indecomposable objects of
$\tilde\A_n$ are parameterised by the elements of
$\tilde\Si_n\setminus\{*\}$. There are canonical recollements of
abelian categories
  \begin{equation*}
\begin{tikzcd}
  \A_1 \arrow[tail]{rr} &&\tilde\A_n
  \arrow[twoheadrightarrow,yshift=-1.5ex]{ll}
  \arrow[twoheadrightarrow,yshift=1.5ex]{ll}
  \arrow[twoheadrightarrow]{rr}[description]{\tilde s_n^i} &&\tilde\A_{n-1}
  \arrow[tail,yshift=-1.5ex]{ll}{\tilde d_n^{i}}
  \arrow[tail,yshift=1.5ex]{ll}[swap]{\tilde d_n^{i+1}}&(0\le i<n)
\end{tikzcd}
\end{equation*}
(with $\tilde d_n^{0}=\tilde d_n^{n}$) and composing them yields a recollement
 \begin{equation*}
\begin{tikzcd}[column sep=large]
  \A_{n-1} \arrow[tail]{rr}[description]{e_n^0} &&\tilde\A_n
  \arrow[twoheadrightarrow,yshift=-1.5ex]{ll}
  \arrow[twoheadrightarrow,yshift=1.5ex]{ll}
  \arrow[twoheadrightarrow]{rr}[description]{\tilde s_2^1\cdots \tilde s_n^{n-1}} &&\tilde\A_{1}.
  \arrow[tail,yshift=-1.5ex]{ll}
  \arrow[tail,yshift=1.5ex]{ll}
\end{tikzcd}
\end{equation*}
Furthermore, there are equivalences
\[t_n^i\colon\tilde\A_n\longiso\tilde\A_n\qquad (0\le i<n)\]
which are given by a cyclic permutation $S_j\mapsto S_{j+i}$ of the
simple representations. We obtain a correspondence between standard morphisms in
$\tilde\Si_n$ and exact functors:
\[\tilde\d_n^i\longleftrightarrow \tilde d_n^i\qquad
  \tilde\s_n^i\longleftrightarrow \tilde s_n^i\qquad
  \e_n^i\longleftrightarrow  e_n^i\qquad  \t_n^i\longleftrightarrow  t_n^i\qquad (0\le i<n).
\]

The following result generalises Theorem~\ref{th:abelian}.  As before,
we consider $k$-linear abelian categories together with $k$-linear
exact functors, up to natural isomorphism, that admit a homological
factorisation.

\begin{thm}\label{th:uniserial}
  Let $k$ be a field. The assignments $\Si_n\mapsto\A_n$ and
  $\tilde\Si_n\mapsto\tilde\A_n$ provide an equivalence between the
  enlarged category of finite strings and the category of $k$-linear abelian
  categories satisfying \emph{(Ab2)--(Ab5)}.
\end{thm}
\begin{proof}
  We adapt the proof of Theorem~\ref{th:abelian}.  Any $k$-linear
  abelian category satisfying (Ab2)--(Ab5) decomposes into a finite
  coproduct of connected abelian categories, which are (up to an
  equivalence) of the form $\A_n$ or $\tilde\A_n$, respectively, by
  Lemmas~\ref{le:uniserial} and \ref{le:uniserial-cyc}.  Thus we
  obtain a bijection between the isomorphism classes of objects. It
  remains to consider the morphisms, and we may restrict ourselves to
  connected categories.  An exact functor $F\colon\A\to\B$ which
  admits a homological factorisation sends indecomposable objects
  either to indecomposables or to zero; see
  Lemma~\ref{le:exact-fun}. This yields a morphism
  $\p\colon\Si_\A\to\Si_\B$ between the corresponding strings, given
  by $M_{\p(s)}=F(M_s)$ for each $s\in\Si_\A$. The assignment
  $F\mapsto\p$ is injective since $k$-linear exact functors which
  admit a homological factorisation are naturally isomorphic when they
  coincide on indecomposable objects; see Lemma~\ref{le:standard}. The
  assignment is surjective, by Lemmas~\ref{le:decomp} and
  \ref{le:decomp-cyc}, in combination with
  Lemma~\ref{le:factorisation}.
\end{proof}

\section{Non-crossing partitions of type $B$}

We wish to describe the subobjects of $\tilde\Si_n$ in the enlarged category of
finite strings. This description is parallel to that for $\Si_n$ and
involves the non-crossing partitions of type $B$.

Let $S\subseteq\Si_\bbZ$. We call $S$ \emph{thick} if $*\in S$ and for any
pair $s,t\in S$ of non-zero strings we have $st\in S$, and moreover
$(s', t'-1), (t',s''), (s''+1,t'')\in S$ provided that
$s'\le t'\le s''\le t''$. We denote by $\Thick(S)$ the smallest thick
subset of $\Si_\bbZ$ containing $S$.

A set $S\subseteq \Si_{\bbZ}$ of non-zero strings
is called \emph{non-crossing} provided that $s,t\in S$ and
$s'\le t'\le s''\le t''$ implies $s=t$.

For $n>0$ consider the canonical projection
$p\colon\Si_\bbZ\to\tilde\Si_n$.  Then a subset
$S\subseteq\tilde\Si_n$ is \emph{thick} if $p^{-1}(S)$ is thick,
and $S$ is \emph{non-crossing} if $p^{-1}(S)$ is non-crossing. Note
that $\ell(s)\le n$ for any $s^{(n)}$ when $S$ is non-crossing.

\begin{lem}
The assignment $S\mapsto\Thick(S)$ gives a bijection between the
non-crossing subsets and the thick subsets of $\tilde\Si_n$.
\end{lem}
\begin{proof}
  The inverse map takes a thick subset $T\subseteq\tilde\Si_n$ to the unique
  non-crossing subset $S\subseteq T$ with $\Thick(S)=T$.
\end{proof}

For non-crossing subsets $S,S'$ of $\tilde\Si_n$ we set
\[S\le S'\quad :\iff\quad \Thick(S)\subseteq\Thick(S').\] This yields
the structure of a poset. In fact, the non-crossing subsets form a
lattice since the thick subsets of $\tilde\Si_n$ are closed under
intersections. We denote this lattice by $\NC(\tilde\Si_n)$.

Let $n\in\bbN$. We consider the set
\[[2n]=\{0,1,\cdots,n-1,\bar 0,\bar 1,\cdots,\overline{n-1}\},\] where
$\bar x$ is identified with $x+n$ for $0\le x<n$, and
$\bar{\bar x}:=x$.  For a partition $P=(P_\a)$ of $[2n]$ we require
that $\overline{P_\a}$ is a part of $P$ for each $\a$.  Each partition
is determined by the corresponding set of strings
$S(P)\subseteq\tilde\Si_{n}$, where by definition
$s^{(n)}\in \tilde\Si_{n}$ with $0\le s'<n$ belongs to $S(P)$ if for
some $\a$ we have $s',s''\in P_\a$ and $i\not\in P_\a$ for all
$s'<i\le s''$. The partitions of $[2n]$ are partially ordered via
refinement, and the non-crossing partitions then form a lattice which
is denoted by $\NC^B(n)$; cf.\ \cite{Re1997,Si2000}.

\begin{lem}\label{le:nc-B}
 There is a lattice isomorphism
  $\NC(\tilde\Si_{n})\iso\NC^B(n)$ which is given by  $S\mapsto P(S)$.
\end{lem}
\begin{proof}
  Adapt the proof of Lemma~\ref{le:nc}.
\end{proof}

\begin{rem}
  Let $P=(P_\a)$ be a non-crossing partition of $[2n]$ and denote by $S=(S_\a)$
  the corresponding partition of $S=S(P)$. Then $P$ has at most one part $P_\a$
  satisfying $\overline{P_\a}=P_\a$. In fact,  $\overline{P_\a}=P_\a$ holds
  if and only if $\Thick(S_\a)$ is infinite.
\end{rem}

As before, we write $P(S)$ for the partition of $[2n]$ corresponding
to a non-crossing set $S\subseteq\tilde\Si_{n}$. For a monomorphism
$\p\colon X\to\tilde\Si_n$ in the category of finite strings we set
\[S(\p):=\{\p(s)\mid s\in X \text{ simple}\}.\]

\begin{thm}\label{th:nc-B}
  Let $n\in\bbN$. The subobjects of $\tilde\Si_n$ in the enlarged
  category of finite strings form a lattice which is canonically
  isomorphic to the lattice of non-crossing partitions
  $\NC^B(n)$. The isomorphism sends a monomorphism
  $\p\colon X\to\tilde\Si_n$ to $P(S(\p))$.
\end{thm}
\begin{proof}
  Adapt the proof of Theorem~\ref{th:sub}.
\end{proof}

\section{Thick subcategories}

Results about subobjects in categories of strings correspond to
statements about thick subcategories of abelian categories, because of
the correspondence from Theorem~\ref{th:uniserial}.

Recall that a full subcategory of an abelian category is \emph{thick}
if it is closed under direct summands and the two out of three
property holds for any short exact sequence.

%The following observation provides the link between subobjects
%of abelian categories and thick subcategories.

\begin{lem}\label{le:thick}
  Let $k$ be a field and let $\A$ be a $k$-linear abelian category
  satisfying \emph{(Ab2)--(Ab5)}. Then every thick subcategory of $\A$
  satisfies again \emph{(Ab2)--(Ab5)}.
\end{lem}
\begin{proof}
  Let $\C\subseteq\A$ be a thick subcategory. Then $\C$ is closed
  under images of morphisms in $\C$ because $\A$ is hereditary. It
  follows that the category $\C$ is abelian and again
  hereditary. Also, $\C$ is necessarily a length category.  If
  $X\in\C$ is simple, then $\End(X)$ is isomorphic to $k[t]/(t^p)$ for
  some $p\ge 1$, since $X$ is indecomposable in $\A$. Schur's lemma
  then implies $p=1$. It remains to show that $\C$ is a uniserial
  category with finitely many simple objects. We may assume that
  either $\A=\A_n$ or $\A=\tilde\A_n$ for some $n\in\bbN$. Then a
  representative set of simple objects in $\C$ identifies with a
  non-crossing subset $S$ in $\Si_n$ or $\tilde\Si_n$. The set $S$ is
  finite since the length of any string in $S$ is bounded by $n$.  Let
  $M_s,M_t$ be simple objects in $\C$ corresponding to strings
  $s,t\in S$. Then $\Ext^1(M_t,M_s)\neq 0$ iff $st\neq *$. It is clear
  that for each $s\in S$ there is at most one $t\in S$ with
  $st\neq *$, and dually there is at most one $r\in S$ with
  $rs\neq *$. Then a criterion from \cite{AR1968} implies that $\C$ is
  uniserial.
\end{proof}

Now we can deduce classifications of thick subcategories from
Theorems~\ref{th:sub} and \ref{th:nc-B}. The first part is due to
Ingalls and Thomas \cite{IT2009} and only included for completeness; the
second part seems to be new.

\begin{cor}
  Let $k$ be a field and $n\in\bbN$. 
  \begin{enumerate}
  \item There is a canonical isomorphism between the lattice of thick
    subcategories of $\A_n$ and the lattice $\NC(n+1)$.
  \item There is a canonical isomorphism between the lattice of thick
    subcategories of $\tilde\A_n$ and the lattice $\NC^B(n)$.
    \end{enumerate}
\end{cor}
\begin{proof}
  We apply Lemma~\ref{le:thick}. From the homological factorisation of
  an exact functor it follows that each subobject of $\A_n$ or
  $\tilde\A_n$ is given by the inclusion of a thick subcategory. On
  the other hand, all thick subcategories arise in this
  way. Theorem~\ref{th:uniserial} provides the correspondence with
  subobjects of $\Si_n$ and $\tilde\Si_n$, respectively. Then the
  assertion follows for $\A_n$ from Theorem~\ref{th:sub} and for
  $\tilde\A_n$ from Theorem~\ref{th:nc-B}.
\end{proof}

\begin{rem}
  (1) The classification of thick subcategories for abelian categories
  of the form $\A_n$ or $\tilde\A_n$ given by a field and $n\in\bbN$
  generalises to any connected hereditary and uniserial length
  category with finitely many isomorphism classes of simple
  objects. The proof is essentially the same, because indecomposable
  objects can be identified with strings which encode their
  composition series. Then thick subcategories correspond bijectively
  to non-crossing sets of strings.

(2) The category of regular modules over a tame hereditary algebra is
an example of an hereditary and uniserial length category
\cite{DR1976}. For the module category of a tame hereditary algebra
one can show that any thick subcategory is contained in the thick
subcategory of regular modules, provided it is not generated by an
exceptional sequence \cite{Di2009}. This yields a classification of
all thick subcategories, complementing the work in
\cite{HK2016,IS2010,IT2009}.

(3) For an hereditary abelian category $\A$,  thick subcategories of the
bounded derived category $\bfD^b(\A)$ correspond bijectively to thick
subcategories of $\A$ via
\[\bfD^b(\A)\supseteq\C\longmapsto \{H^0(X)\in\A\mid
  X\in\C\}\subseteq\A,\]
see \cite[Proposition 4.4.17]{Kr2021}.
\end{rem}

\appendix

\section{Basic strings}

The properties of basic strings and the connection with linear representations become
more transparent if we consider analogues of face and degeneracy maps for the poset
of integers. We view this poset as a category, and this
means that morphisms $\bbZ\to\bbZ$  are viewed as functors.  For $i\in\bbZ$
we define morphisms
\[\d^i\colon\bbZ\lto\bbZ,\qquad j\mapsto\begin{cases}j&j<i\\ j+1&j\ge
    i\end{cases}\]
  and
\[\s^i\colon\bbZ\lto\bbZ,\qquad j\mapsto\begin{cases}j&j\le i\\ j-1&j>
    i.\end{cases}\] These satisfy the simplicial identities
\eqref{eq:rel}. Moreover, they are related via adjunctions:
\[\cdots \dashv \d^{i+1}\dashv\s^i\dashv\d^i \dashv\s^{i-1}\dashv\cdots\]

Let $k$ be a field. Then for each  $i\in\bbZ$  precomposition with $\d^i$ and
$\s^i$ yields exact functors
\[
  \begin{tikzcd}
\Rep(\bbZ^\op,k) \arrow[rr,yshift=-.75ex,"(\d^{i})^*",swap]
 &&\Rep(\bbZ^\op,k). \arrow[ll,yshift=.75ex,"(\s^i)^*",swap] 
\end{tikzcd} 
\]
This assignment is contravariant and therefore reverses the directions of
functors. Thus the dual simplicial identities but the same adjunctions
\[\cdots \dashv (\d^{i+1})^*\dashv(\s^i)^*\dashv(\d^i)^* \dashv(\s^{i-1})^*\dashv\cdots\]
are satisfied.

For $i\in\bbZ$ let $S_i$ denote the simple representation concentrated
in $i$, that is, $S_i(j)= 0$ for all $j\neq i$. Let
$n\in\bbN$. Then precomposition with the inclusion $[n]\to\bbZ$ yields
an exact functor
\[\Rep(\bbZ^\op,k)\lto \Rep([n]^\op,k)=\A_n\] which becomes an
equivalence when restricted to the full subcategory of objects in
$\Rep(\bbZ^\op,k)$ with composition factors in
$\{S_0,\ldots,S_{n-1}\}$. Viewing this as an identification, the
functors $(\d^i)^*$ and $(\s^{i-1})^*$ restrict to exact functors
\[\A_{n}\xto{\ (\d^i)^*\ }\A_{n-1} \qquad (0\le i<n)\] 
and
\[\A_{n-1}\xto{\ (\s^{i-1})^*\ }\A_{n} \qquad (0\le i \le n).\]

Recall that $\Si_\bbZ$ denotes the set of basic strings. Then
$\Si_\bbZ\setminus\{*\}$ identifies with the indecomposable objects of
finite length in $\Rep(\bbZ^\op,k)$ via $s\mapsto M_s$, as in
Lemma~\ref{le:lin-alg}. This identification yields maps
$\Si_\bbZ\to\Si_\bbZ$ which are induced by $(\d^i)^*$ and $(\s^i)^*$,
respectively. Restricting these maps for any $n\in\bbN$ to the set
$\Si_n$ of basic strings with composition factors in
$\{s_0,\ldots,s_{n-1}\}$ gives
\[\s_n^i=(\d^i)^*|_{\Si_n} \qquad\text{and}\qquad \d_{n}^i=(\s^{i-1})^*|_{\Si_{n-1}}.\]
Then the following elementary observation (reflecting a duality for
the simplicial category $\De$, cf.\ \cite[VIII.7]{MM1994}) explains
the simplicial relations (and any further properties) for
$\d^i_n\colon\Si_{n-1}\to\Si_n$ and $\s^i_n\colon\Si_{n}\to\Si_{n-1}$.

\begin{lem}
  Consider symbols $(\d^i,\s^i)$ and $(d^i,s^i)$ for some integers
  $i\in\bbZ$. After substituting  $\d^i\mapsto s^i$ and $\s^i\mapsto
  d^{i+1}$ and reversing the order of composition, the identities
  \eqref{eq:rel} hold for  $(d^i,s^i)$ if and only if they hold for
  $(\d^i,\s^i)$.\qed
\end{lem}


\begin{thebibliography}{10}

\bibitem{AR1968} I. Kr. Amdal\ and\ F. Ringdal, Cat\'egories
  unis\'erielles, C. R. Acad. Sci. Paris S\'er. A-B {\bf 267} (1968),
  A85--A87 and A247--A249.

\bibitem{BR1987} M. C. R. Butler\ and\ C. M. Ringel, Auslander--Reiten
  sequences with few middle terms and applications to string algebras,
  Comm. Algebra {\bf 15} (1987), no.~1-2, 145--179.

\bibitem{Ca1985} P. Cartier, Homologie cyclique: rapport sur des
  travaux r\'{e}cents de Connes, Karoubi, Loday, Quillen$\ldots$,
  Ast\'{e}risque No. 121-122 (1985), 123--146.

\bibitem{Co1983} A. Connes, Cohomologie cyclique et foncteurs ${\rm
    Ext}\sp n$, C. R. Acad. Sci. Paris S\'{e}r. I Math. {\bf 296}
  (1983), no.~23, 953--958.

\bibitem{Di2009} N. D. Dichev, Thick subcategories for quiver
  representations, PhD thesis, Universit\"at Paderborn, 2009.
  
\bibitem{DR1976} V. Dlab\ and\ C. M. Ringel, Indecomposable
  representations of graphs and algebras, Mem. Amer. Math. Soc. {\bf
    6} (1976), no.~173, {\rm v}+57 pp.

\bibitem{DJW2019} T. Dyckerhoff, G. Jasso\ and\ T. Walde, Simplicial
  structures in higher Auslander-Reiten theory, Adv. Math. {\bf 355}
  (2019), 106762, 73 pp.
  
\bibitem{Ga1962} P. Gabriel, Des cat\'egories ab\'eliennes,
  Bull. Soc. Math. France {\bf 90} (1962), 323--448.

\bibitem{GZ1967} P. Gabriel\ and\ M. Zisman, {\it Calculus of
    fractions and homotopy theory}, Springer-Verlag New York, Inc.,
  New York, 1967.

\bibitem{HJ2021} M. Herschend\ and\ P. J\o rgensen, Classification of
  higher wide subcategories for higher Auslander algebras of type $A$,
  J. Pure Appl. Algebra {\bf 225} (2021), no.~5, Paper No. 106583, 22
  pp.
  
\bibitem{HK2016} A. Hubery\ and\ H. Krause, A categorification of
  non-crossing partitions, J. Eur. Math. Soc. (JEMS) {\bf 18} (2016),
  no.~10, 2273--2313.

\bibitem{IS2010} K. Igusa\ and\ R. Schiffler, Exceptional sequences
  and clusters, J. Algebra {\bf 323} (2010), no.~8, 2183--2202.

\bibitem{IT2009} C. Ingalls\ and\ H. Thomas, Noncrossing partitions
  and representations of quivers, Compos. Math. {\bf 145} (2009),
  no.~6, 1533--1562.

\bibitem{I2011} O. Iyama, Cluster tilting for higher Auslander
  algebras, Adv. Math. {\bf 226} (2011), no.~1, 1--61.
  
\bibitem{Kr2021} H. Krause, {\it Homological Theory of Representations},
Cambridge Studies in Advanced Mathematics, 195, Cambridge University
Press, Cambridge, 2021.
  
\bibitem{Kr1972} G. Kreweras, Sur les partitions non crois\'ees d'un
  cycle, Discrete Math. {\bf 1} (1972), no.~4, 333--350.

\bibitem{ML1998} S. Mac Lane, {\it Categories for the working
    mathematician}, second edition, Graduate Texts in Mathematics, 5,
  Springer-Verlag, New York, 1998.

\bibitem{MM1994} S. Mac Lane\ and\ I. Moerdijk, {\it Sheaves in geometry
  and logic}, corrected reprint of the 1992 edition, Universitext,
Springer-Verlag, New York, 1994.

\bibitem{Re1997} V. Reiner, Non-crossing partitions for classical
  reflection groups, Discrete Math. {\bf 177} (1997), no.~1-3,
  195--222.

\bibitem{Ri1984} C. M. Ringel, {\it Tame algebras and integral
    quadratic forms}, Lecture Notes in Mathematics, 1099,
  Springer-Verlag, Berlin, 1984.

\bibitem{Si2000} R. Simion, Noncrossing partitions, Discrete
  Math. {\bf 217} (2000), no.~1-3, 367--409.
  
\bibitem{Ri2016} C. M. Ringel, The Catalan combinatorics of the
  hereditary Artin algebras, in {\it Recent developments in
    representation theory}, 51--177, Contemp. Math., 673,
  Amer. Math. Soc., Providence, RI, 2016.
  
\end{thebibliography}
\end{document}